\documentclass[11pt,a4paper]{article}
\usepackage{amsfonts}
\usepackage{latexsym}
\usepackage{cite}
\usepackage{amsmath,amsfonts,latexsym,amssymb}
\usepackage[mathscr]{eucal}
\usepackage{cases,color}
\usepackage{amsthm}
\usepackage[bf,small]{caption2}
\usepackage{float}
\usepackage{graphicx}
\usepackage{amsmath}
\usepackage{amssymb}
\usepackage[all]{xy}

\newtheorem{theorem}{theorem}[section]

\newtheorem{cla}[theorem]{Claim}

\newtheorem{lem}[theorem]{Lemma}

\newtheorem{rmk}[theorem]{Remark}
\newtheorem{thm}[theorem]{Theorem}

\begin{document}

\title{\vspace{-2cm}\textbf{The ${\rm SL}(2,\mathbb{C})$-character variety of the magic $3$-manifold}}
\author{\Large Haimiao Chen}

\date{}
\maketitle

\begin{abstract}
  We determine the irreducible ${\rm SL}(2,\mathbb{C})$-character variety of the 3-chain link exterior which is called the `magic $3$-manifold', and deduce a formula for the twisted Alexander polynomial associated to each ${\rm SL}(2,\mathbb{C})$-representation.

\end{abstract}

\section{Introduction}

Throughout, let $G={\rm SL}(2,\mathbb{C})$. Given a finitely presented group $\Gamma$, call a homomorphism $\rho:\Gamma\to G$ a
{\it $G$-representation} of $\Gamma$.
Define the {\it character} of $\rho$ as the function $\chi_\rho:\Gamma\to\mathbb{C}$, $x\mapsto{\rm tr}(\rho(x))$.
Call $\rho$ reducible if elements of ${\rm Im}(\rho)$ have a common eigenvector; otherwise, call $\rho$ irreducible.
As a well-known fact, two irreducible representations $\rho,\rho'$ have the same character if and only if they are conjugate, meaning that there exists $\mathbf{a}\in G$ such that $\rho'(x)=\mathbf{a}\rho(x)\mathbf{a}^{-1}$ for all $x\in\Gamma$.
Call $\hom(\Gamma,G)$ the {\it $G$-representation variety} of $\Gamma$ and denote it by $\mathcal{R}(\Gamma)$.
The set $\mathcal{X}(\Gamma)=\{\chi_\rho\colon\rho\in\mathcal{R}(\Gamma)\}$ can be defined by finitely many polynomials, and is called the
$G$-{\it character variety} of $\Gamma$. The subset $\mathcal{X}^{\rm irr}(\Gamma)$ consisting of characters of irreducible representations is Zariski open, and is called the {\it irreducible character variety}.

Character variety has been playing a significant role in low-dimensional topology.
For a link $L\subset S^3$, let $E_L$ denote its exterior; let $\pi(L)=\pi_1(E_L)$.
Abbreviate $\mathcal{R}(\pi(L))$ to $\mathcal{X}(L)$ and call it the $G$-character variety of $L$, and so forth.
In the literature, there have been a lot of results on character varieties of knots, but no explicit result is obtained for links with at least 3-components, except the Borromean link \cite{CY24,HP23,MS22}.

\begin{figure}[H]
  \centering
  \includegraphics[width=4cm]{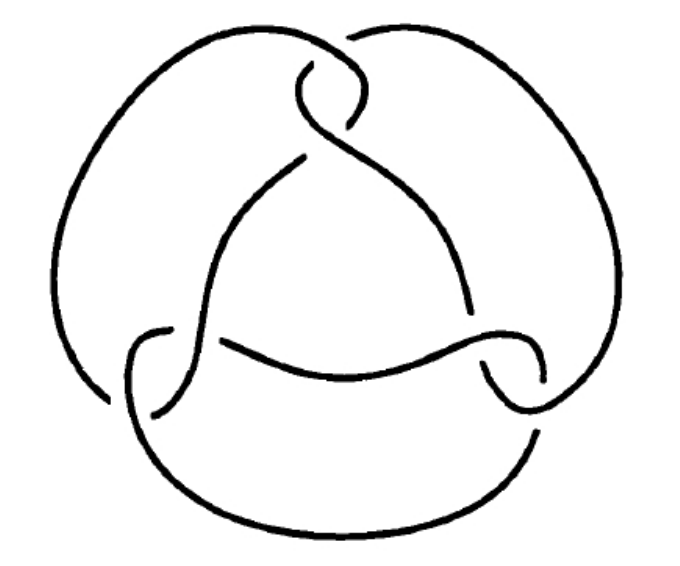}\\
  \caption{The 3-chain link $C$.}\label{fig:magic}
\end{figure}

In this paper, we focus on the 3-chain link (denoted by $C$ and shown in Figure \ref{fig:magic}), which is also the $(2,2,2)$-pretzel link. Its exterior $E_C$ is hyperbolic, as known to Thurston \cite{Th80}, and was called the `magic 3-manifold' by Gordon and Wu \cite{Go98,GW99}, due to that by Dehn filling, one may get many hyperbolic manifolds with small volumes, as well as some important examples of exceptional fillings of cusped hyperbolic manifolds. Non-hyperbolic Dehn fillings had been clarified by Martelli and Petronio \cite{MP06}.

For such an interesting link, it is worth working out the $G$-character variety.
We find that the irreducible character variety $\mathcal{X}^{\rm irr}(C)$ has 9 irreducible components; see Theorem \ref{thm:main}. In particular, we explicitly describe the canonical component, i.e. the one containing the character of a lift of the holonomy representation $\pi(C)\to{\rm PSL}(2,\mathbb{C})$.
Moreover, we deduce an elegant formula for the twisted Alexander polynomial associated to each representation (see Theorem \ref{thm:TAP}), which is an important invariant.

\section{The character variety}

\subsection{Preparation}\label{sec:preparation}

We use boldface letters to denote $2\times 2$ matrices, which are possibly not in $G$.
Let $\mathbf{e}$ denote the $2\times 2$ identity matrix.

For $\mathbf{x}\in G$, let ${\rm Cen}(\mathbf{x})=\{\mathbf{u}\in G\colon \mathbf{x}\mathbf{u}=\mathbf{u}\mathbf{x}\}$.
As explained on \cite[Page 6]{CY24}, when $\mathbf{x}\ne\pm\mathbf{e}$, each element of ${\rm Cen}(\mathbf{x})$ has the form $\alpha\mathbf{x}+\beta\mathbf{e}$.

For the following two paragraphs, refer to \cite[Section 3.3]{CY24}.

Given $\mathbf{x}_1,\mathbf{x}_2\in G$, they share an eigenvector if and only if ${\rm tr}([\mathbf{x}_1,\mathbf{x}_2])=2$, where the commutator
$[\mathbf{x}_1,\mathbf{x}_2]=\mathbf{x}_1\mathbf{x}_2\mathbf{x}_1^{-1}\mathbf{x}_2^{-1}$.
Let $t_1={\rm tr}(\mathbf{x}_1)$, $t_2={\rm tr}(\mathbf{x}_2)$, $t_{12}={\rm tr}(\mathbf{x}_1\mathbf{x}_2)$.
Then
\begin{align}
{\rm tr}([\mathbf{x}_1,\mathbf{x}_2])=t_{12}^2-t_1t_2t_{12}+t_1^2+t_2^2-2.    \label{eq:tr(commutator)}
\end{align}
When ${\rm tr}[\mathbf{a}_1,\mathbf{a}_2]\ne 2$, up to (simultaneous) conjugacy $(\mathbf{a}_1,\mathbf{a}_2)$ is determined by $t_1,t_2,t_{12}$.

Let $F_3=\langle x_1,x_2,x_3\mid-\rangle$, the free group on $x_1,x_2,x_3$. It is known that through
$\chi\mapsto t_{i_1\cdots i_r}=\chi(x_{i_1\cdots i_r})$ for $1\le i_1<\cdots<i_r\le 3$, the character variety $\mathcal{X}(F_3)$ is isomorphic to
\begin{align}
\big\{(t_1,t_2,t_3,t_{12},t_{13},t_{23}, t_{123})\colon t_{123}^2-\nu_1t_{123}+\nu_0=0\big\}\subset\mathbb{C}^7,   \label{eq:F3}
\end{align}
for certain $\nu_0,\nu_1\in\mathbb{Z}[t_1,t_2,t_3,t_{12},t_{13},t_{23}]$.
So the character variety of any 3-generator group can be embedded into the hypersurface given by (\ref{eq:F3}).

For any $2\times 2$ matrix $\mathbf{a}$ over any commutative ring, by Cayley-Hamilton Theorem,
$\mathbf{a}^2={\rm tr}(\mathbf{a})\mathbf{a}-\det(\mathbf{a})\mathbf{e}$.
Let $\mathbf{a}^\ast$ denote the adjoint matrix of $\mathbf{a}$. Then $\mathbf{a}^\ast={\rm tr}(\mathbf{a})\mathbf{e}-\mathbf{a}$.
In particular, $\mathbf{a}^\ast=\mathbf{a}^{-1}$ if $\det(\mathbf{a})=1$.

\begin{lem}\label{lem:det}
For any $2\times 2$ matrices $\mathbf{a},\mathbf{b}$,
\begin{align*}
\mathbf{a}\mathbf{b}+\mathbf{b}\mathbf{a}&={\rm tr}(\mathbf{a})\mathbf{b}+{\rm tr}(\mathbf{b})\mathbf{a}
-{\rm tr}(\mathbf{a}\mathbf{b}^\ast)\mathbf{e},  \\
\det(\mathbf{a}+\mathbf{b})&=\det(\mathbf{a})+\det(\mathbf{b})+{\rm tr}(\mathbf{a}\mathbf{b}^\ast).
\end{align*}
In particular,
$\det(\mathbf{e}+\mathbf{a})=1+\det(\mathbf{a})+{\rm tr}(\mathbf{a}).$
\end{lem}

\begin{proof}
Since
\begin{align*}
{\rm tr}(\mathbf{a}\mathbf{b})\mathbf{e}&=\mathbf{a}\mathbf{b}+(\mathbf{a}\mathbf{b})^\ast
=\mathbf{a}\mathbf{b}+\mathbf{b}^\ast\mathbf{a}^\ast=\mathbf{a}\mathbf{b}
+({\rm tr}(\mathbf{b})\mathbf{e}-\mathbf{b})({\rm tr}(\mathbf{a})\mathbf{e}-\mathbf{a})  \\
&=\mathbf{a}\mathbf{b}+\mathbf{b}\mathbf{a}-{\rm tr}(\mathbf{a})\mathbf{b}-{\rm tr}(\mathbf{b})\mathbf{a}
+{\rm tr}(\mathbf{a}){\rm tr}(\mathbf{b})\mathbf{e},
\end{align*}
and ${\rm tr}(\mathbf{a}){\rm tr}(\mathbf{b})-{\rm tr}(\mathbf{a}\mathbf{b})={\rm tr}(\mathbf{a}\mathbf{b}^\ast)$, we have the first identity.

To show the second identity, we proceed as
\begin{align*}
\det(\mathbf{a}+\mathbf{b})\cdot\mathbf{e}&={\rm tr}(\mathbf{a}+\mathbf{b})\cdot(\mathbf{a}+\mathbf{b})-(\mathbf{a}+\mathbf{b})^2  \\
&={\rm tr}(\mathbf{a}+\mathbf{b})\cdot(\mathbf{a}+\mathbf{b})-\mathbf{a}^2-\mathbf{b}^2-(\mathbf{a}\mathbf{b}+\mathbf{b}\mathbf{a})  \\
&={\rm tr}(\mathbf{a}+\mathbf{b})\cdot(\mathbf{a}+\mathbf{b})-\big({\rm tr}(\mathbf{a})\mathbf{a}-\det(\mathbf{a})\mathbf{e}\big)  \\
&\ \ \ \ -\big({\rm tr}(\mathbf{b})\mathbf{b}-\det(\mathbf{b})\mathbf{e}\big)
-\big({\rm tr}(\mathbf{a})\mathbf{b}+{\rm tr}(\mathbf{b})\mathbf{a}-{\rm tr}(\mathbf{a}\mathbf{b}^\ast)\mathbf{e}\big)  \\
&=(\det(\mathbf{a})+\det(\mathbf{b})+{\rm tr}(\mathbf{a}\mathbf{b}^\ast))\cdot\mathbf{e}.
\end{align*}
\end{proof}

\subsection{From matrix equations to trace equations}

\begin{figure}[H]
  \centering
  \includegraphics[width=4cm]{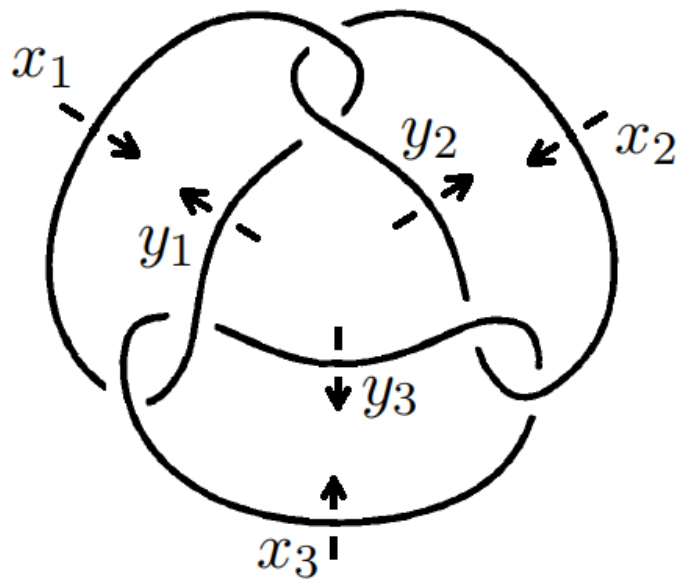}\\
  \caption{Generators of $\pi(C)$.}\label{fig:generator}
\end{figure}

Referred to Figure \ref{fig:generator}, a Wirtinger presentation of $\pi(C)$ is
\begin{align*}
\big\langle x_1,x_2,x_3,y_1,y_2,y_3\mid\ &y_1=y_2x_1y_2^{-1}, \ y_2=y_3x_2y_3^{-1}, \\
&y_1=x_3^{-1}x_1x_3, \ y_2=x_1^{-1}x_2x_1,\ y_3=x_2^{-1}x_3x_2\big\rangle.
\end{align*}
With $y_1,y_2,y_3$ substituted, this can be transformed into
\begin{align}
\big\langle x_1,x_2,x_3\mid [x_1,x_3x_1^{-1}x_2]=1, \ [x_2,x_1x_2^{-1}x_3]=1\rangle.   \label{eq:presentation}
\end{align}

Given $\underline{\mathbf{x}}=(\mathbf{x}_1,\mathbf{x}_2,\mathbf{x}_3)\in G^3$, sufficient and necessary conditions for there to exist $\rho\in\mathcal{R}(C)$ with $\rho(x_i)=\mathbf{x}_i$, $i=1,2,3$ are
\begin{align}
[\mathbf{x}_1,\mathbf{x}_3\mathbf{x}_1^{-1}\mathbf{x}_2]&=\mathbf{e},  \label{eq:relation-1}   \\
[\mathbf{x}_2,\mathbf{x}_1\mathbf{x}_2^{-1}\mathbf{x}_3]&=\mathbf{e}.  \label{eq:relation-2}
\end{align}
When these hold, $\rho$ is unique; denote it by $\rho_{\underline{\mathbf{x}}}$.

Suppose $\rho_{\underline{\mathbf{x}}}$ is irreducible.

\begin{lem}\label{lem:non-commuting}
$\mathbf{x}_i\mathbf{x}_j\ne\mathbf{x}_j\mathbf{x}_i$ for any $i\ne j$. In particular, $\mathbf{x}_1,\mathbf{x}_2,\mathbf{x}_3\ne\pm\mathbf{e}$.
\end{lem}
\begin{proof}
If $\mathbf{x}_1\mathbf{x}_2=\mathbf{x}_2\mathbf{x}_1$, then (\ref{eq:relation-2}) would imply $\mathbf{x}_2\mathbf{x}_3=\mathbf{x}_3\mathbf{x}_2$, contradicting the irreducibility of $\rho_{\underline{\mathbf{x}}}$. Hence $\mathbf{x}_1\mathbf{x}_2\ne\mathbf{x}_2\mathbf{x}_1$.

Similarly, $\mathbf{x}_1\mathbf{x}_3\ne\mathbf{x}_3\mathbf{x}_1$, and $\mathbf{x}_2\mathbf{x}_3\ne\mathbf{x}_3\mathbf{x}_2$.
\end{proof}

Let $t_i={\rm tr}(\mathbf{x}_i)$, $t_{123}={\rm tr}(\mathbf{x}_1\mathbf{x}_2\mathbf{x}_3)$.
For $i\ne j$, let
$$t_{ij}={\rm tr}(\mathbf{x}_i\mathbf{x}_j), \qquad
r_{ij}={\rm tr}(\mathbf{x}_i\mathbf{x}_j^{-1})={\rm tr}(\mathbf{x}_i^{-1}\mathbf{x}_j)=t_it_j-t_{ij}.$$

As a simple observation, (\ref{eq:relation-1}), (\ref{eq:relation-2}) are respectively equivalent to
\begin{align}
\mathbf{x}_3\mathbf{x}_1^{-1}\mathbf{x}_2&=\alpha\mathbf{x}_1+\beta\mathbf{e},    \label{eq:relation-1'}  \\
\mathbf{x}_1\mathbf{x}_2^{-1}\mathbf{x}_3&=\alpha'\mathbf{x}_2+\beta'\mathbf{e},     \label{eq:relation-2'}
\end{align}
for some $\alpha,\beta,\alpha',\beta'$. 
Using
$$\mathbf{x}_1\mathbf{x}_2^{-1}\mathbf{x}_1=(\mathbf{x}_1\mathbf{x}_2^{-1})^2\mathbf{x}_2
=(r_{12}\mathbf{x}_1\mathbf{x}_2^{-1}-\mathbf{e})\mathbf{x}_2
=r_{12}\mathbf{x}_1-\mathbf{x}_2,$$
and similarly $\mathbf{x}_2\mathbf{x}_1^{-1}\mathbf{x}_2=r\mathbf{x}_2-\mathbf{x}_1$,
we can rewrite (\ref{eq:relation-1'}), (\ref{eq:relation-2'}) as
\begin{align}
\mathbf{x}_3&=(\alpha\mathbf{x}_1+\beta\mathbf{e})\mathbf{x}_2^{-1}\mathbf{x}_1
=\alpha\mathbf{x}_1\mathbf{x}_2^{-1}\mathbf{x}_1+\beta\mathbf{x}_2^{-1}\mathbf{x}_1  \nonumber  \\
&=\alpha(r_{12}\mathbf{x}_1-\mathbf{x}_2)+\beta(t_2\mathbf{e}-\mathbf{x}_2)\mathbf{x}_1  \nonumber  \\
&=(\alpha r_{12}+\beta t_2)\mathbf{x}_1-\alpha\mathbf{x}_2-\beta\mathbf{x}_2\mathbf{x}_1,  \label{eq:rewrite-1}   \\
\mathbf{x}_3&=\mathbf{x}_2\mathbf{x}_1^{-1}(\alpha'\mathbf{x}_2+\beta'\mathbf{e})=\alpha'\mathbf{x}_2\mathbf{x}_1^{-1}\mathbf{x}_2
+\beta'\mathbf{x}_2\mathbf{x}_1^{-1} \nonumber  \\
&=\alpha'(r_{12}\mathbf{x}_2-\mathbf{x}_1)+\beta'\mathbf{x}_2(t_1\mathbf{e}-\mathbf{x}_1)  \nonumber  \\
&=(\alpha'r_{12}+\beta't_1)\mathbf{x}_2-\alpha'\mathbf{x}_1-\beta'\mathbf{x}_2\mathbf{x}_1.  \nonumber  
\end{align}
Hence
$$(\beta'-\beta)\mathbf{x}_2\mathbf{x}_1+(\alpha r_{12}+\beta t_2+\alpha')\mathbf{x}_1-(\alpha+\alpha'r_{12}+\beta't_1)\mathbf{x}_2=0.$$
If $\beta\ne \beta'$, then we can write $\mathbf{x}_2(\mathbf{x}_1-a\mathbf{e})=b\mathbf{x}_1$ for some $a,b$.
By Lemma \ref{lem:non-commuting}, $b\ne 0$, so
$\det(\mathbf{x}_1-a\mathbf{e})=b^2\ne 0$, but this implies $\mathbf{x}_2=b\mathbf{x}_1(\mathbf{x}_1-a\mathbf{e})^{-1}\in{\rm Cen}(\mathbf{x}_1)$,
a contradiction to Lemma \ref{lem:non-commuting}.

Thus,
$\beta'=\beta$. Consequently, $\alpha'=-\alpha r_{12}-\beta t_2$, and $\alpha=-\alpha'r_{12}-\beta't_1$.
Eliminating $\alpha'$ yields
\begin{align}
(r_{12}^2-1)\alpha=(t_1-t_2r_{12})\beta.  \label{eq:coefficient}
\end{align}
Taking traces of both sides of (\ref{eq:rewrite-1}) yields
\begin{align}
(t_1r_{12}-t_2)\alpha+r_{12}\beta=t_3.   \label{eq:trace-1}
\end{align}
By Lemma \ref{lem:det}, $\det(\alpha\mathbf{x}_1+\beta\mathbf{e})=1$ is equivalent to
\begin{align}
\alpha^2+t_1\alpha\beta+\beta^2=1.  \label{eq:det}
\end{align}

By (\ref{eq:rewrite-1}),
\begin{align}
r_{13}&={\rm tr}\big(\mathbf{x}_1^{-1}((r_{12}\alpha+t_2\beta)\mathbf{x}_1-\alpha\mathbf{x}_2-\beta\mathbf{x}_2\mathbf{x}_1)\big)  \nonumber  \\
&=2(r_{12}\alpha+t_2\beta)-r_{12}\alpha-t_2\beta=r_{12}\alpha+t_2\beta,   \label{eq:t13-0}  \\
r_{23}&={\rm tr}\big(\mathbf{x}_2^{-1}((r_{12}\alpha+t_2\beta)\mathbf{x}_1-\alpha\mathbf{x}_2-\beta\mathbf{x}_2\mathbf{x}_1)\big)  \nonumber  \\
&=(r_{12}\alpha+t_2\beta)r_{12}-2\alpha-t_1\beta\stackrel{(\ref{eq:coefficient})}=-\alpha.   \label{eq:t23-0}
\end{align}
Since $\mathbf{x}_3\mathbf{x}_1^{-1}\mathbf{x}_2=\mathbf{x}_3(t_1\mathbf{e}-\mathbf{x}_1)\mathbf{x}_2
=t_1\mathbf{x}_3\mathbf{x}_2-\mathbf{x}_3\mathbf{x}_1\mathbf{x}_2$, we have
\begin{align}
{\rm tr}(\mathbf{x}_3\mathbf{x}_1^{-1}\mathbf{x}_2)=t_1t_{23}-t_{123}=t_1t_2t_3-t_1r_{23}-t_{123}.    \label{eq:trace-2}
\end{align}
By (\ref{eq:relation-1'}), ${\rm tr}(\mathbf{x}_3\mathbf{x}_1^{-1}\mathbf{x}_2)=t_1\alpha+2\beta$,
which combined with (\ref{eq:t23-0}) yields
\begin{align}
t_{123}=t_1t_2t_3-2\beta.    \label{eq:t123}
\end{align}

It follows from (\ref{eq:relation-1}), (\ref{eq:relation-2}) that
$[\mathbf{x}_1^{-1},\mathbf{x}_2]=[\mathbf{x}_2^{-1},\mathbf{x}_3]=[\mathbf{x}_3^{-1},\mathbf{x}_1]$;
let $\mathbf{g}$ denote the common value, and let $\eta={\rm tr}(\mathbf{g})+3$.
Then by (\ref{eq:tr(commutator)}),
\begin{align}
\eta=r_{ij}^2-t_it_jr_{ij}+t_i^2+t_j^2+1, \qquad  1\le i<j\le 3.  \label{eq:eta-0}
\end{align}
In the case $i=2$, $j=3$, by (\ref{eq:t23-0}),
\begin{align}
\eta=\alpha^2+t_2t_3\alpha+t_2^2+t_3^2+1.   \label{eq:eta}
\end{align}

From (\ref{eq:rewrite-1}) we see that the irreducibility of $\rho_{\underline{\mathbf{x}}}$ is equivalent to that $\mathbf{x}_1,\mathbf{x}_2$ share no eigenvector. 
This is equivalent to $\eta\ne 5$, as we always assume.

\begin{rmk}
\rm Our goal is to find $\alpha,\beta,r_{12}$ satisfying (\ref{eq:coefficient})--(\ref{eq:det}), for any given $t_1,t_2,t_3$. As long as $\eta\ne 5$ is fulfilled, up to conjugacy $(\mathbf{x}_1,\mathbf{x}_2)$ is fixed, and then $\mathbf{x}_3$ is determined by (\ref{eq:rewrite-1}). This determines the conjugacy class of $\rho_{\underline{\mathbf{x}}}$.

It will be helpful to bear in the mind that $C$ is symmetric under the $(2\pi/3)$-rotation, and moreover, by (\ref{eq:t123}), $\beta$ is invariant under the rotation.
\end{rmk}

\subsection{Solving the trace equations}


If $t_1\alpha+\beta=0$, then (\ref{eq:coefficient})--(\ref{eq:det}) are equivalent to
\begin{align*}
\alpha\in\{\pm1\}, \qquad \beta=-\alpha t_1, \qquad  t_3=-\alpha t_2, \qquad  r_{12}^2-t_1t_2r_{12}+t_1^2=1.
\end{align*}
Remember (\ref{eq:t23-0}) that $r_{23}=-\alpha$. Rewrite (\ref{eq:t13-0}) as $r_{13}=t_1t_3+\alpha r_{12}$.
Moreover, the condition $\eta\ne 5$ is equivalent to $t_2^2\ne 3$.

From now on, suppose $t_1\alpha+\beta\ne 0$. Then (\ref{eq:trace-1}) becomes
\begin{align}
r_{12}=\frac{t_2\alpha+t_3}{t_1\alpha+\beta};   \label{eq:r}
\end{align}
substituting it into (\ref{eq:coefficient}), we obtain
\begin{align}
0&=\alpha(t_2\alpha+t_3)^2-(\alpha+t_1\beta)(t_1\alpha+\beta)^2+t_2\beta(t_2\alpha+t_3)(t_1\alpha+\beta)  \nonumber  \\
&=\alpha(t_2\alpha+t_3)^2-(t_1+\alpha\beta)(t_1\alpha+\beta)+t_2(t_2\alpha+t_3)(1-\alpha^2)  \nonumber   \\
&=(t_2t_3-t_1\beta)\alpha^2+(t_3^2-t_1^2-\beta^2+t_2^2)\alpha+t_2t_3-t_1\beta  \nonumber   \\
&=(t_2t_3-t_1\beta)(1-\beta^2-t_1\alpha\beta)+(t_3^2-t_1^2-\beta^2+t_2^2)\alpha+t_2t_3-t_1\beta  \nonumber  \\
&=\big((t_1^2-1)\beta^2-t_1t_2t_3\beta+t_2^2+t_3^2-t_1^2\big)\alpha+(2-\beta^2)(t_2t_3-t_1\beta).   \label{eq:alpha-0}
\end{align}

Suppose
\begin{align}
(t_1^2-1)\beta^2-t_1t_2t_3\beta+t_2^2+t_3^2-t_1^2=0.    \label{eq:vanish}
\end{align}
Then $(2-\beta^2)(t_2t_3-t_1\beta)=0$.
\begin{enumerate}
  \item If $t_1\beta=t_2t_3$, then $\beta^2=t_2^2+t_3^2-t_1^2$, so $(t_1^2-t_2^2)(t_1^2-t_3^2)=0$.
  \begin{itemize}
    \item If $t_1^2=t_2^2$, then $t_2=\kappa t_1$ with $\kappa\in\{\pm1\}$, and $\beta=\kappa t_3$. By (\ref{eq:r}), $r_{12}=\kappa$.
          By (\ref{eq:t13-0}), $r_{13}=t_1t_3+\kappa\alpha$, so by (\ref{eq:t23-0}),
          $r_{23}=\kappa(t_1t_3-r_{13}).$
          Moreover,
          \begin{align*}
          r_{13}^2-t_1t_3r_{13}+t_3^2=\alpha^2+t_2t_3\alpha+t_3^2=\alpha^2-t_1\alpha\beta+\beta^2\stackrel{(\ref{eq:det})}=1.
          \end{align*}
          As is easy to see, $\eta\ne 5$ is equivalent to $t_1^2\ne 3$.
    \item If $t_1^2=t_3^2$, then $t_3=\kappa t_1$ with $\kappa\in\{\pm1\}$, and $\beta=\kappa t_2$. Due to (\ref{eq:t23-0}),
          we may rewrite (\ref{eq:det}) as
          $$r_{23}^2-t_2t_3r_{23}+t_2^2=1.$$
          Furthermore, as can be verified,
          \begin{align*}
          &r_{12}\stackrel{(\ref{eq:r})}=\frac{t_2\alpha+t_3}{t_1\alpha+\beta}=t_1t_2+\kappa\alpha=t_1t_2-\kappa r_{23}, \\
          &r_{13}=r_{12}\alpha+t_2\beta=(t_1t_2+\kappa\alpha)\alpha+\kappa t_2^2=\kappa.
          \end{align*}
          As is easy to see, $\eta\ne 5$ is equivalent to $t_3^2\ne 3$.
  \end{itemize}
  \item If $t_1\beta\ne t_2t_3$, then $\beta^2=2$, i.e. $\beta\in\{\pm\sqrt{2}\}$, and (\ref{eq:vanish}) becomes
        \begin{align}
        t_1^2+t_2^2+t_3^2-2=t_1t_2t_3\beta.   \label{eq:beta}
        \end{align}
        If $t_3\beta=t_1t_2$, then $t_1^2+t_2^2+t_3^2-2=2t_3^2$, so that
        $$2(t_1^2+t_2^2-2)=2t_3^2=(t_3\beta)^2=t_1^2t_2^2,$$
        implying $t_1^2=2$ or $t_2^2=2$; respectively, $t_2^2=t_3^2$ or $t_1^2=t_3^2$. Either case can be incorporated into one of the previous cases.
        Thus we can assume $t_3\beta\ne t_1t_2$. By symmetry, we can also assume $t_2\beta\ne t_1t_3$.

        It follows from (\ref{eq:eta}), (\ref{eq:det}) that
        $\eta=(t_2t_3-t_1\beta)\alpha+t_2^2+t_3^2,$
        hence
        \begin{align}
        \alpha=\frac{\eta-t_2^2-t_3^2}{t_2t_3-t_1\beta}.  \label{eq:alpha-2}
        \end{align}
        Then (\ref{eq:det}) becomes
        \begin{align}
        \eta^2-(t_1^2+t_2^2+t_3^2+2)\eta+t_1^2t_2^2+t_2^2t_3^2+t_1^2t_3^2+4=0.   \label{eq:eta-2}
        \end{align}
        From (\ref{eq:t23-0}) we immediately se
        \begin{align*}
        r_{23}=\frac{\eta-t_2^2-t_3^2}{t_1\beta-t_2t_3}.
        \end{align*}
        By rotational symmetry, we have
        $$r_{12}=\frac{\eta-t_1^2-t_2^2}{t_3\beta-t_1t_2},  \qquad  r_{13}=\frac{\eta-t_1^2-t_3^2}{t_2\beta-t_1t_3}.$$
        Alternatively, these can be obtained from (\ref{eq:r}), (\ref{eq:t13-0}) with moderate efforts, using (\ref{eq:beta}), (\ref{eq:alpha-2}), (\ref{eq:eta-2}).
\end{enumerate}

\medskip

Now suppose
$(t_1^2-1)\beta^2-t_1t_2t_3\beta+t_2^2+t_3^2-t_1^2\ne 0.$
Then (\ref{eq:alpha-0}) becomes
\begin{align}
\alpha=\frac{(\beta^2-2)(t_2t_3-t_1\beta)}{(t_1^2-1)\beta^2-t_1t_2t_3\beta+t_2^2+t_3^2-t_1^2},   \label{eq:alpha-3}
\end{align}
which combined with (\ref{eq:det}) implies
\begin{align}
0&=(\beta^2-2)^2(t_2t_3-t_1\beta)^2+(\beta^2-1)\big((t_1^2-1)\beta^2-t_1t_2t_3\beta+t_2^2+t_3^2-t_1^2\big)^2  \nonumber  \\
&\ \ \ \ +t_1\beta(\beta^2-2)(t_2t_3-t_1\beta)\big((t_1^2-1)\beta^2-t_1t_2t_3\beta+t_2^2+t_3^2-t_1^2\big)   \nonumber  \\
&=\beta^6-t_1t_2t_3\beta^5+(\sigma_2-2\sigma_1-1)\beta^4+(8-\sigma_1)t_1t_2t_3\beta^3  \nonumber   \\
&\ \ \ \ +(\sigma_1^2+\sigma_3-4\sigma_2+2\sigma_1)\beta^2-8t_1t_2t_3\beta+4\sigma_2-\sigma_1^2,   \label{eq:canonical}
\end{align}
where Let
$\sigma_1=t_1^2+t_2^2+t_3^2$, $\sigma_2=t_1^2t_2^2+t_1^2t_3^2+t_2^2t_3^2$, $\sigma_3=t_1^2t_2^2t_3^2.$

By (\ref{eq:alpha-3}), the condition $t_1\alpha+\beta\ne 0$ is equivalent to
\begin{align}
\beta^3-\sigma_1\beta+2t_1t_2t_3\ne 0.   \label{eq:nonvanishing}
\end{align}
We can rewrite (\ref{eq:r}) as
\begin{align}
r_{12}=\frac{t_1t_2\beta^3+(1-t_1^2-t_2^2)t_3\beta^2+(t_3^2-2)t_1t_2\beta+(\sigma_1-2t_3^2)t_3}{\beta^3-\sigma_1\beta+2t_1t_2t_3}.  \label{eq:r12-3}
\end{align}
By rotational symmetry or by computation using (\ref{eq:t13-0}) (\ref{eq:t23-0}),
\begin{align}
r_{13}&=\frac{t_1t_3\beta^3+(1-t_1^2-t_3^2)t_2\beta^2+(t_2^2-2)t_1t_3\beta+(\sigma_1-2t_2^2)t_2}{\beta^3-\sigma_1\beta+2t_1t_2t_3},
\label{eq:r13-3}   \\
r_{23}&=\frac{t_2t_3\beta^3+(1-t_2^2-t_3^2)t_1\beta^2+(t_1^2-2)t_2t_3\beta+(\sigma_1-2t_1^2)t_1}{\beta^3-\sigma_1\beta+2t_1t_2t_3}.  \label{eq:r23-3}
\end{align}

If $\beta^2=2$, then by (\ref{eq:alpha-3}), $\alpha=0$, which contradicts (\ref{eq:det}). Hence $\beta^2\ne 2$.

If $\alpha \ne0$, then
\begin{align*}
\eta&\stackrel{(\ref{eq:eta-0})}=r_{12}^2-t_1t_2r_{12}+t_1^2+t_2^2+1
\stackrel{(\ref{eq:coefficient})}=\frac{\beta(t_1-t_2r_{12})}{\alpha}-t_1t_2r_{12}+t_1^2+t_2^2+2  \\
&\ =\frac{t_1\beta}{\alpha}+t_1^2+t_2^2+2-\frac{t_1\alpha+\beta}{\alpha}t_2r_{12}
\stackrel{(\ref{eq:r})}=t_1^2+2+\frac{t_1\beta-t_2t_3}{\alpha} \\
&\stackrel{(\ref{eq:alpha-3})}=\frac{3\beta^2+t_1t_2t_3\beta-\sigma_1-4}{\beta^2-2};
\end{align*}
if $\alpha=0$, then (\ref{eq:coefficient})--(\ref{eq:det}) become $t_1=t_2r_{12}$, $t_3=\beta r_{12}$, $\beta^2=1$, which imply
$$\eta=\sigma_1-t_1t_2t_3\beta+1=\frac{3\beta^2+t_1t_2t_3\beta-\sigma_1-4}{\beta^2-2}.$$
Thus, the condition $\eta\ne 5$ is equivalent to
\begin{align}
2\beta^2-t_1t_2t_3\beta+\sigma_1-6\ne 0.    \label{eq:nonvanishing-2}
\end{align}

\begin{cla}
If {\rm(\ref{eq:canonical})}, {\rm(\ref{eq:nonvanishing})} hold, then $(t_i^2-1)\beta^2-t_1t_2t_3\beta+\sigma_1-2t_i^2=0$ for at least one $i$.
\end{cla}

\begin{proof}
Assume $(t_i^2-1)\beta^2-t_1t_2t_3\beta+\sigma_1-2t_i^2=0$ for each $i$. Then
$$t_i^2=\frac{\beta^2+t_1t_2t_3\beta-\sigma_1}{\beta^2-2}=:t^2.$$
Write $t_1t_2t_3=\epsilon t^3$, with $\epsilon\in\{\pm1\}$.
Then $(\beta t^2-\beta-\epsilon t)(\beta-\epsilon t)=0$. In view of (\ref{eq:nonvanishing}), $\beta\ne \epsilon t$, so
$\beta=\epsilon t/(t^2-1)$.
Substituting this into (\ref{eq:canonical}) yields
$$t^4(t^2-2)^4(2t^2-1)^2=0.$$
Hence $t^2\in\{0,2,1/2\}$. But this contradicts (\ref{eq:nonvanishing}).
\end{proof}

Thus, the case $(t_1^2-1)\beta^2-t_1t_2t_3\beta+\sigma_1-2t_1^2=0$ can be included: when it holds, we have
$(t_i^2-1)\beta^2-t_1t_2t_3\beta+\sigma_1-2t_i^2\ne 0$ for $i=2$ or $i=3$, so by symmetry we can still deduce (\ref{eq:canonical}).

\begin{rmk}
\rm If $t_1,t_2,t_3\in\{\pm2\}$, then $\sigma_1=12$, $\sigma_2=48$, $\sigma_3=64$, so (\ref{eq:canonical}) becomes
$(\beta-2\epsilon)^3(\beta-3\epsilon)(\beta^2+\epsilon\beta+2)=0$, with $\epsilon=t_1t_2t_3/8\in\{\pm1\}$.
By (\ref{eq:nonvanishing}), (\ref{eq:nonvanishing-2}), $\beta\ne 2\epsilon,3\epsilon$, hence $\beta^2+\epsilon\beta+2=0$.

When $t_1=t_2=t_3=2$ and $\beta=(-1-\sqrt{-7})/2$, we can compute
$$r_{12}=r_{13}=r_{23}=\frac{4\beta^3-14\beta^2+8\beta+8}{\beta^3-12\beta+16}=\frac{1-\sqrt{-7}}{2}, \qquad  t_{123}=9+\sqrt{-7}.$$
This is consistent with the hyperbolic structure given in \cite[Example 6.8.2]{Th80}.
\end{rmk}

\subsection{The result}

By the facts presented in Section \ref{sec:preparation}, a general element of $\mathcal{X}^{\rm irr}(C)$ can be identified with a tuple $(t_1,t_2,t_3,r_{12},r_{13},r_{23},\beta)$, with $\beta=(t_1t_2t_3-t_{123})/2$, as indicated in (\ref{eq:t123}).
Let
$$\sigma_1=t_1^2+t_2^2+t_3^2, \qquad  \sigma_2=t_1^2t_2^2+t_1^2t_3^2+t_2^2t_3^2, \qquad \sigma_3=t_1^2t_2^2t_3^2.$$

\begin{thm}\label{thm:main}
The irreducible character variety of $C$ is decomposed as
$$\mathcal{X}^{\rm irr}(C)=\Big({\bigcup}_{i=1}^3\mathcal{X}^+_{1,i}\Big)
\cup\Big({\bigcup}_{i=1}^3\mathcal{X}^-_{1,i}\Big)\cup\mathcal{X}_2^+\cup\mathcal{X}_2^-\cup\mathcal{X}_3,$$
where
\begin{itemize}
  \item $\mathcal{X}_{1,i}^{\pm}$ consists of $(t_1,t_2,t_3,r_{12},r_{13},r_{23},\beta)$ such that
        \begin{align*}
        t_{i+1}=\pm t_{i-1}\notin\{\sqrt{3},-\sqrt{3}\}, \qquad \beta=\pm t_i, \qquad  r_{i-1,i+1}=\pm1,   \\
        r_{i-1,i}=t_{i-1}t_i\mp r_{i,i+1}, \qquad  r_{i,i+1}^2-t_it_{i+1}r_{i,i+1}+t_i^2=1;
        \end{align*}
  \item $\mathcal{X}_2^{\pm}$ consists of $(t_1,t_2,t_3,r_{12},r_{13},r_{23},\beta)$ with
        \begin{align*}
        r_{12}=\frac{\eta-t_1^2-t_2^2}{t_3\beta-t_1t_2},  \qquad
        r_{13}=\frac{\eta-t_1^2-t_3^2}{t_2\beta-t_1t_3},  \qquad
        r_{23}=\frac{\eta-t_2^2-t_3^2}{t_1\beta-t_2t_3},
        \end{align*}
        for $t_1,t_2,t_3,\beta,\eta$ satisfying
        \begin{align*}
        \beta=\pm\sqrt{2}, \qquad \sigma_1-2=t_1t_2t_3\beta, \qquad \eta^2-(\sigma_1+2)\eta+\sigma_2+4=0, \\
        \eta\ne 5, \qquad  t_1\beta\ne t_2t_3, \qquad t_2\beta\ne t_1t_3, \qquad t_3\beta\ne t_1t_2.
        \end{align*}
  \item $\mathcal{X}_3$ consists of $(t_1,t_2,t_3,r_{12},r_{13},r_{23},\beta)$ with $r_{12}, r_{13}, r_{23}$ respectively given by
        {\rm(\ref{eq:r12-3})}--{\rm(\ref{eq:r23-3})}, for $t_1,t_2,t_3,\beta$ satisfying
        \begin{align*}
        &\beta^6-t_1t_2t_3\beta^5+(\sigma_2-2\sigma_1-1)\beta^4+(8-\sigma_1)t_1t_2t_3\beta^3  \\
        &\ \ \ \ \ \ +(\sigma_1^2+\sigma_3-4\sigma_2+2\sigma_1)\beta^2-8t_1t_2t_3\beta+4\sigma_2-\sigma_1^2=0, \\
        &\beta^2\ne 2, \qquad  \beta^3-\sigma_1\beta+2t_1t_2t_3\ne 0, \qquad  2\beta^2-t_1t_2t_3\beta+\sigma_1-6\ne 0.
        \end{align*}
\end{itemize}
\end{thm}

In total, $\mathcal{X}^{\rm irr}(C)$ consists of $9$ irreducible components.
The Zariski closure of $\mathcal{X}_3$ is the canonical component, which is 3-dimensional, and each other component has dimension $2$.

\section{Twisted Alexander polynomial}

Recall the definition of twisted Alexander polynomial and some basic facts, which were presented at the beginning of \cite[Section 4]{CY24}.

For a ring $R$, let $\mathcal{M}_n(R)$ denote the ring of $n\times n$ matrices over $R$.

Suppose $L$ is an oriented link, with components $K_1,\ldots,K_m$.
Let
$$\pi:=\pi(L)\cong\langle x_1,\ldots,x_\ell\mid r_1,\ldots,r_{\ell-1}\rangle$$
be a presentation strongly Tietz equivalent to some Wirtinger presentation, such that each $x_j$ comes from an arc of $K_{\sigma(j)},\sigma(j)\in\{1,\ldots,m\}$.
Let $F_\ell$ denote the free group generated by $x_1,\ldots,x_\ell$.
Let $M$ be the $(\ell-1)\times\ell$ matrix whose $(i,j)$-entry is the image of $\partial r_i/\partial x_j$ (the Fox derivative) under the ring homomorphism $q:\mathbb{Z}[F_\ell]\to\mathbb{Z}[\pi]$
induced by the canonical quotient map $F_\ell\twoheadrightarrow\pi$, and let $M_v\in\mathcal{M}_{\ell-1}(\mathbb{Z}[\pi])$ be the matrix obtained from deleting the $v$-th column of $M$.
Let
$\mathfrak{a}:\pi\to \mathbb{Z}^{\oplus m}=\langle s_1\rangle\oplus\cdots\oplus\langle s_m\rangle$
denote the abelianization map, which sends $x_j$ to $s_{\sigma(j)}$.

Given a representation $\rho:\pi\to{\rm GL}(n,\mathbb{C})$, extend the composite
$$\pi\stackrel{\mathfrak{a}\times\rho}\longrightarrow\mathbb{Z}^{\oplus m}\times {\rm GL}(n,\mathbb{C})\hookrightarrow \mathbb{Z}[s_1^{\pm1},\ldots,s_m^{\pm1}]\times\mathcal{M}_n(\mathbb{C})\to\mathcal{M}_n(\mathbb{C}[s_1^{\pm1},\ldots,s_m^{\pm1}])$$
by linearity to a ring homomorphism
\begin{align*}
\Phi_\rho:\mathbb{Z}[\pi]\to\mathcal{M}_n(\mathbb{C}[s_1^{\pm1},\ldots,s_m^{\pm1}]).
\end{align*}
The {\it twisted Alexander polynomial} of $L$ associated to $\rho$ is defined to be
$$\Delta_{L}^{\rho}=\Delta_{L}^{\rho}(s_1,\ldots,s_m)\doteq\frac{\det\Phi_\rho(M_v)}{\det\Phi_\rho(1-x_v)}\in \mathbb{C}(s_1,\ldots,s_m),$$
where $\Phi_\rho(M_v)\in\mathcal{M}_{n(\ell-1)}(\mathbb{C}[s_1^{\pm1},\ldots,s_m^{\pm1}])$ is the big matrix obtained from $M_v$ by replacing each entry with its image under $\Phi_\rho$, and $\doteq$ means an equality up to multiplication by $s_1^{k_1}\cdots s_m^{k_m}$ for $k_1,\ldots,k_m\in\mathbb{Z}$.


\medskip


The remainder of this section is devoted to proving
\begin{thm}\label{thm:TAP}
Given a representation $\rho:\pi(C)\to G$, let $t_i={\rm tr}(\rho(x_i))$, $r_{ij}={\rm tr}(\rho(x_i^{-1}x_j))$,
$t_{123}={\rm tr}(\rho(x_1x_2x_3))$, then
\begin{align*}
\Delta^{\rho}_C
&\doteq t_{123}-t_1t_ 2t_3+\Big(s_1+\frac{1}{s_1}\Big)r_{23}+\Big(s_2+\frac{1}{s_2}\Big)r_{13}+\Big(s_3+\frac{1}{s_3}\Big)r_{12}  \\
&\ \ \ \ -\Big(\frac{s_2}{s_3}+\frac{s_3}{s_2}\Big)t_1-\Big(\frac{s_1}{s_3}+\frac{s_3}{s_1}\Big)t_2-\Big(\frac{s_1}{s_2}+\frac{s_2}{s_1}\Big)t_3  \\
&\ \ \ \ +\frac{s_2s_3}{s_1}+\frac{s_1s_3}{s_2}+\frac{s_1s_2}{s_3}+\frac{s_1}{s_2s_3}+\frac{s_2}{s_1s_3}+\frac{s_3}{s_1s_2}.
\end{align*}
\end{thm}

To keep the expressions compact, for $w\in\mathbb{Z}[F_3]$ we denote $q(w)\in\mathbb{Z}[\pi]$ also by $w$.
With the presentation (\ref{eq:presentation}) which is strongly Tietze equivalent to a Wirtinger presentation,
$$r_1=x_1x_3x_1^{-1}x_2(x_3x_1^{-1}x_2x_1)^{-1}, \qquad  r_2=x_2x_1x_2^{-1}x_3(x_1x_2^{-1}x_3x_2)^{-1}.$$
Using the properties of Fox derivative (see \cite[Page 117]{Li97}) and that when $r_i=fg^{-1}$ with $f,g\in F_3$,
$$\frac{\partial r_i}{\partial x_j}=\frac{\partial}{\partial x_j}(f\overline{g})=\frac{\partial}{\partial x_j}f-f\overline{g}\frac{\partial}{\partial x_j}g=\frac{\partial}{\partial x_j}(f-g) \quad \text{in\ } \mathbb{Z}[\pi],$$
we obtain the following equations in $\mathbb{Z}[\pi]$:
\begin{align*}
\frac{\partial r_1}{\partial x_1}&=\frac{\partial}{\partial x_1}(x_1x_3x_1^{-1}x_2-x_3x_1^{-1}x_2x_1)
=1-x_1x_3x_1^{-1}+x_3x_1^{-1}-x_3x_1^{-1}x_2,  \\
\frac{\partial r_1}{\partial x_3}&=\frac{\partial}{\partial x_3}(x_1x_3x_1^{-1}x_2-x_3x_1^{-1}x_2x_1)
=x_1-1,  \\
\frac{\partial r_2}{\partial x_1}&=\frac{\partial}{\partial x_1}(x_2x_1x_2^{-1}x_3-x_1x_2^{-1}x_3x_2)
=x_2-1,  \\
\frac{\partial r_2}{\partial x_3}&=\frac{\partial}{\partial x_3}(x_2x_1x_2^{-1}x_3-x_1x_2^{-1}x_3x_2)
=x_2x_1x_2^{-1}-x_1x_2^{-1}.
\end{align*}
Hence
\begin{align*}
M_2=\left(\begin{array}{cc} 1-x_1x_3x_1^{-1}+x_3x_1^{-1}-x_3x_1^{-1}x_2 & x_1-1  \\
x_2-1 & x_2x_1x_2^{-1}-x_1x_2^{-1}  \end{array}\right).   
\end{align*}

Let $\mathbf{x}_i=\rho(x_i)$, $i=1,2,3$.
Then
\begin{align*}
\Delta^{\rho}_C&\doteq\frac{\det\Phi_\rho(M_2)}{\det\Phi_\rho(1-x_2)}  \\
&=\frac{1}{\det(\mathbf{e}-s_2\mathbf{x}_2)}
\det\left(\begin{array}{cc}
\mathbf{z} &  s_1\mathbf{x}_1-\mathbf{e}  \\
s_2\mathbf{x}_2-\mathbf{e} &  s_1s_2^{-1}(s_2\mathbf{x}_2-\mathbf{e})\mathbf{x}_1\mathbf{x}_2^{-1} \end{array}\right)   \\
&\doteq\det\left(\begin{array}{cc}
\mathbf{z} &  s_1^{-1}s_2(s_1\mathbf{x}_1-\mathbf{e})\mathbf{x}_2\mathbf{x}_1^{-1}  \\
\mathbf{e} &  \mathbf{e}  \end{array}\right)  \\
&=\det\big(\mathbf{z}-s_1^{-1}s_2(s_1\mathbf{x}_1-\mathbf{e})\mathbf{x}_2\mathbf{x}_1^{-1}\big),
\end{align*}
where
$$\mathbf{z}=\mathbf{e}-s_3\mathbf{x}_1\mathbf{x}_3\mathbf{x}_1^{-1}+s_3s_1^{-1}\mathbf{x}_3\mathbf{x}_1^{-1}
-s_3s_1^{-1}s_2\mathbf{x}_3\mathbf{x}_1^{-1}\mathbf{x}_2.$$
We can write
$\Delta^{\rho}_C\doteq\det(\mathbf{u}\mathbf{v}+\mathbf{w})$, with
$$\mathbf{u}=\mathbf{e}-s_1\mathbf{x}_1, \qquad  \mathbf{v}=\frac{s_2}{s_1}\Big(\mathbf{e}+\frac{s_3}{s_2}\mathbf{x}_3\mathbf{x}_2^{-1}\Big)\mathbf{x}_2\mathbf{x}_1^{-1},   \qquad
\mathbf{w}=\mathbf{e}-\frac{s_2s_3}{s_1}\mathbf{x}_3\mathbf{x}_1^{-1}\mathbf{x}_2.$$
This can be compared with the formula given in \cite[Example 4.4]{Ch21}.

By Lemma \ref{lem:det},
\begin{align}
\Delta^{\rho}_C\doteq\det(\mathbf{u})\det(\mathbf{v})+\det(\mathbf{w})+{\rm tr}(\mathbf{u}\mathbf{v}\mathbf{w}^\ast),  \label{eq:TAP}
\end{align}
and
\begin{align}
\det(\mathbf{u})&=1+s_1^2-s_1t_1,  \qquad
\det(\mathbf{v})=\frac{s_2^2}{s_1^2}\Big(1+\frac{s_3^2}{s_2^2}+\frac{s_3}{s_2}r_{23}\Big),  \label{eq:det(u&v)}  \\
\det(\mathbf{w})&=1+\frac{s_2^2s_3^2}{s_1^2}-\frac{s_2s_3}{s_1}{\rm tr}(\mathbf{x}_3\mathbf{x}_1^{-1}\mathbf{x}_2)  \nonumber  \\
&=1+\frac{s_2^2s_3^2}{s_1^2}-\frac{s_2s_3}{s_1}(t_1t_2t_3-t_1r_{23}-t_{123});   \label{eq:det(w)}
\end{align}
for the last equality, refer to (\ref{eq:trace-2}).

Note that
\begin{align*}
\mathbf{u}\mathbf{v}&=\frac{s_2}{s_1}\mathbf{x}_2\mathbf{x}_1^{-1}+\frac{s_3}{s_1}\mathbf{x}_3\mathbf{x}_1^{-1}
-s_2\mathbf{x}_1\mathbf{x}_2\mathbf{x}_1^{-1}-s_3\mathbf{x}_1\mathbf{x}_3\mathbf{x}_1^{-1},  \\
\mathbf{w}^\ast&=\mathbf{e}-s_2s_3s_1^{-1}\mathbf{x}_2^{-1}\mathbf{x}_1\mathbf{x}_3^{-1}.
\end{align*}
By (\ref{eq:relation-1}), $\mathbf{x}_1^{-1}\mathbf{x}_2^{-1}\mathbf{x}_1\mathbf{x}_3^{-1}=\mathbf{x}_2^{-1}\mathbf{x}_1\mathbf{x}_3^{-1}\mathbf{x}_1^{-1}$.
We can compute
\begin{align}
&{\rm tr}(\mathbf{u}\mathbf{v}\mathbf{w}^\ast)  \nonumber \\
=\ &{\rm tr}(\mathbf{u}\mathbf{v})
-\frac{s_2s_3}{s_1}{\rm tr}\Big(\frac{s_2}{s_1}\mathbf{x}_1\mathbf{x}_3^{-1}\mathbf{x}_1^{-1}
+\frac{s_3}{s_1}\mathbf{x}_3\mathbf{x}_1^{-1}\mathbf{x}_2^{-1}\mathbf{x}_1\mathbf{x}_3^{-1} \nonumber  \\
&\hspace{28mm} -s_2\mathbf{x}_1^2\mathbf{x}_3^{-1}\mathbf{x}_1^{-1}
-s_3\mathbf{x}_1\mathbf{x}_3\mathbf{x}_2^{-1}\mathbf{x}_1\mathbf{x}_3^{-1}\mathbf{x}_1^{-1}\Big) \nonumber  \\
=\ &\frac{s_2}{s_1}r_{12}+\frac{s_3}{s_1}r_{13}-s_2t_2-s_3t_3-\frac{s_2s_3}{s_1}\Big(\frac{s_2}{s_1}t_3+\frac{s_3}{s_1}t_2-s_2r_{13}-s_3r_{12}\Big) \nonumber  \\
=\ &\frac{s_2}{s_1}(1+s_3^2)r_{12}+\frac{s_3}{s_1}(1+s_2^2)r_{13}-s_2\Big(1+\frac{s_3^2}{s_1^2}\Big)t_2-s_3\Big(1+\frac{s_2^2}{s_1^2}\Big)t_3.
\label{eq:trace(uvw)}
\end{align}
Substituting (\ref{eq:det(u&v)})--(\ref{eq:trace(uvw)}) into (\ref{eq:TAP}) and multiplying by $s_1s_2^{-1}s_3^{-1}$ yields the formula given in Theorem \ref{thm:TAP}.



\bigskip

\noindent
Haimiao Chen (orcid: 0000-0001-8194-1264)\ \ \ \ {\it chenhm@math.pku.edu.cn} \\
Department of Mathematics, Beijing Technology and Business University, \\
Liangxiang Higher Education Park, Fangshan District, Beijing, China.

\end{document}